\documentclass[12pt]{amsart}
\usepackage[all]{xy}
\usepackage{a4wide}
\usepackage{amssymb}
\usepackage{caption}
\usepackage{amsthm}
\usepackage{amsmath}
\usepackage{amscd,enumitem}
\usepackage{verbatim}
\usepackage{float}
\usepackage{color}
\usepackage[all]{xy}
\usepackage{hyperref}
\usepackage{cleveref}
\theoremstyle{plain}
\newtheorem{theorem}{Theorem}
\numberwithin{theorem}{section}
\newtheorem{corollary}[theorem]{Corollary}

\newtheorem{lemma}[theorem]{Lemma}

\theoremstyle{definition}
\newtheorem{definition}[theorem]{Definition}
\newtheorem{example}{Example}

\theoremstyle{remark}
\newtheorem*{remark}{Remark}

\newcommand{\C}{\mathbb{C}}

\newcommand{\calP}{\mathcal{P}}

\newcommand{\calS}{\mathcal{S}}

\begin{document}

\title{Partition-theoretic formulas for arithmetic densities, II}

\author{Ken Ono, Robert Schneider and Ian Wagner}

\address{Department of Mathematics, University of Virginia,
Charlottesville, VA 22904} \email{ken.ono691@virginia.edu}

\address{Department of Mathematics, University of Georgia,
Athens, GA 30602}\email{robert.schneider@uga.edu}

\address{Department of Mathematics, Vanderbilt University,
Nashville, TN 37240}
\email{ian.c.wagner@vanderbilt.edu}

\thanks{The first author thanks the support of the Thomas Jefferson Fund and the NSF (DMS-1601306 and DMS-2002265).
}
\subjclass[2010]{05A17, 11P82}

\begin{abstract} 
In earlier work generalizing a 1977 theorem of Alladi, the authors proved a {\it partition-theoretic} formula to compute arithmetic densities of certain subsets of the positive integers $\mathbb N$ as limiting values of $q$-series as $q\to \zeta$ a root of unity (instead of using the usual Dirichlet series to compute densities), replacing multiplicative structures of $\mathbb N$ by analogous structures in the integer partitions $\mathcal P$. 
In recent work, Wang obtains a wide generalization of Alladi's original theorem, in which arithmetic densities of subsets of prime numbers are computed as values of Dirichlet series
arising from Dirichlet convolutions. Here the authors prove that Wang's extension has a
partition-theoretic analogue as well, yielding new $q$-series density formulas for any subset of $\mathbb N$. To do so, we outline a theory of $q$-series density calculations from first principles, based on a statistic we call the ``$q$-density'' of a given subset. This theory in turn yields infinite families of further formulas for arithmetic densities. 
\end{abstract}

\maketitle

\centerline{\it In honor of Srinivasa Ramanujan on the 100th anniversary of his passing}
\  

\section{Introduction and statement of results}\label{Sect1}

It is well-known that the M\"obius function satisfies
\begin{equation}\label{First}
-\lim_{N\rightarrow \infty} \sum_{n=2}^{N} \frac{\mu(n)}{n} = 1.
\end{equation}
This fact can be used to obtain formulas
for arithmetic densities for suitable subsets of the positive integers.
A celebrated example of this idea was obtained by  Alladi  \cite{Alladi1}, who proved for $r,t,\in \mathbb N$ that if $\gcd(r, t) = 1$, then
\begin{equation}\label{AlladiSum}
-\sum_{\substack{ n \geq 2 \\ p_{\rm{min}}(n) \equiv r \pmod{t}}} \frac{\mu(n)}{n} = \frac{1}{\varphi(t)}. 
\end{equation}
Here $\varphi(t)$ is Euler's phi function, and $p_{\rm{min}}(n)$ is the smallest prime factor of $n$.  

Over the last few years, (\ref{AlladiSum}) has been generalized in various ways
in works by Dawsey \cite{Dawsey}, Sweeting and Woo \cite{SweetingWoo}, Kural, McDonald and Sah \cite{KMS},  and Wang \cite{Wang1, Wang2}. In a recent preprint \cite{Wang3}, Wang obtained a beautiful generalization
of this phenomenon which makes use of Dirichlet convolutions. 
If $a: \mathbb{N} \to \C$ is an arithmetic function satisfying $a(1) =1$ and $\sum_{n \geq 2} \frac{|a(n)|}{n} \log \log n < \infty$, then for suitable subsets $\mathcal{S}$ of the positive integers Wang proves that
\begin{equation}\label{WangSum}
- \lim_{N \to \infty} \sum_{\substack{2 \leq n \leq N \\ p_{{\rm{min}}}(n) \in \mathcal{S}}} \frac{(\mu * a)(n)}{n} = d_{\mathcal{S}},
\end{equation}
where $(f * g)(n) = \sum_{d |n} f(d)g(n/d)$ is the classical Dirichlet convolution of $f$ and $g$.

In earlier work the authors obtained a partition-theoretic analogue of formulas such as (\ref{AlladiSum})
using natural analogies between the multiplicative structure of the integers and
the additive structure of partitions.  The purpose of this note is to show that an analogue also exists for (\ref{WangSum}), as a consequence of a general theorem offering infinitely many formulas (see Theorem \ref{Thm1}).

We first recall fundamental features of these analogies.
A \textit{partition} is a finite non-increasing sequence of positive integers, say $\lambda = (\lambda_{1}, \lambda_{2},..., \lambda_{\ell(\lambda)})$, where $\ell(\lambda)$ denotes the {\it length} (number of parts) of $\lambda$.  The {\it size} of $\lambda$ is $\vert \lambda \vert := \lambda_{1} + \lambda_{2} + \cdot \cdot \cdot + \lambda_{\ell(\lambda)}$ (sum of parts) and the {\it norm} of the partition is $N(\lambda) = \lambda_{1} \cdot \lambda_{2} \cdots \lambda_{\ell(\lambda)}$ (product of parts).\footnote{The norm is referred to as the ``integer'' of $\lambda$ and written $n_{\lambda}$ in some earlier works \cite{OnoRolenS, Robert_zeta}.}  Furthermore, we let $\rm{sm}(\lambda) := \lambda_{\ell(\lambda)}$ denote the {\it smallest part} of $\lambda$ (resp. $\rm{lg}(\lambda) := \lambda_{1}$ the {\it largest part} of $\lambda$).

We will use the {\it partition M\"{o}bius function}
\begin{equation}\label{OnePointFour}
\mu_{\mathcal{P}}(\lambda) := \begin{cases} 0 & \rm{if} \ \lambda \ \rm{has \ repeated \ parts}, \\
(-1)^{\ell(\lambda)} & \rm{otherwise}. \end{cases}  
\end{equation}
Notice that $\mu_{\mathcal{P}}(\lambda) = 0$ if $\lambda$ has any repeated parts, which is analogous to the vanishing of $\mu(n)$ for integers $n$ which are not square-free.  In particular, the different parts in a partition $\lambda$ play the role of prime factors of $n$ in this analogy.  We define $\mu_{\mathcal{P}}^{*}(\lambda):= - \mu_{\mathcal{P}}(\lambda)$ as in Dawsey's theorem, to eliminate minus signs in our formulas. Finally, we recall the usual {\it $q$-Pochhammer symbol} $(a;q)_{n}:=\prod_{k=0}^{n-1}(1-aq^k)$ and $(a;q)_{\infty}:=\lim_{n\to \infty}(a;q)_{n}$ for $a,q\in \mathbb C, |q|<1$. 

To complete the analogy, also recall the {\it partition phi function
} defined in \cite{Schneider_arithmetic}, an analogue of Euler's totient function,
\begin{equation*}
\varphi_{\mathcal{P}}(\lambda) := N(\lambda)\cdot  \prod_{\substack{k \in \lambda \\ \text{without repetition}}} \left(1 - k^{-1} \right),
\end{equation*}  
where ``$k\in \lambda$'' means $k\in \mathbb N$ is a part of $\lambda \in \mathcal P$. 

Table 1 offers some of the identifications making up this multiplicative-additive analogy.
\bigskip

\begingroup\label{T1}
\begin{center}
\begin{tabular}{ | c | c | }
\hline
{\bf Natural number $n$} & {\bf Partition $\lambda$}\\ \hline
Prime factors of $n$ & Parts of $\lambda$ \\ \hline
Square-free integers & Partitions into distinct parts \\ \hline
$\mu(n)$ & $\mu_{\mathcal{P}}(\lambda)$ \\ \hline
$\varphi(n)$ & $\varphi_{\mathcal{P}}(\lambda)$ \\ \hline
$p_{\rm{min}}(n)$ & $\rm{sm}(\lambda)$ \\ \hline
$p_{\rm{max}}(n)$ & $\rm{lg}(\lambda)$ \\ \hline
$n^{-s}$ & $q^{\vert \lambda \vert}$ \\ \hline
${\zeta(s)^{-1}}$ & ${(q;q)_{\infty}}$ \\  \hline
$s\to1$ & $q\rightarrow 1$\\
\hline
\end{tabular}
\smallskip
\captionof{table}{\textit{Analogies between arithmetic and partition theory}}
\end{center}
\endgroup

\begin{remark}
We note that other, different multiplicative-additive analogies exist, such as those drawn in \cite{SchneiderSills} between partitions of length $m$ and those of size $m$, in which the duality instead correlates $\zeta(s)$ to the geometric series $q(1-q)^{-1}$, with $m^{-s}$ in bijection with $q^m$. 
\end{remark}

To describe our results, let
$\calS$ be a subset of the positive integers with arithmetic density
\begin{equation*}
d_{\calS}:=\lim_{N \to \infty} \frac{\# \{ n \in \calS : n \leq N \}}{N}.\end{equation*}
As noted above, we let $\mu_{\calP}^*(\lambda):=-\mu_{\calP}(\lambda).$
In \cite{Paper1} the function
\begin{equation}\label{OnePointNine}
F_{\mathcal{S}}(q) := \sum_{\substack{ \lambda \in \mathcal{P} \\ \rm{sm}(\lambda) \in \mathcal{S}}} \mu_{\mathcal{P}}^{*}(\lambda) q^{\vert \lambda \vert}
\end{equation}
was defined, and it was shown, for suitable $\mathcal{S}$ and associated roots of unity $\zeta$, that
\begin{equation}\label{Fdens}
\lim_{q \to \zeta} F_{\mathcal{S}}(q) = d_{\mathcal{S}}.
\end{equation}
This theorem is a partition-theoretic analogue of (\ref{AlladiSum}).
 
\begin{remark}We note that throughout this paper, ``$q\to \zeta$'' with $\zeta$ a root of unity (including the case $\zeta=1$) denotes that $q$ approaches the point on the unit circle from within the unit disk. \end{remark}
 
To state the partition analogue of Wang's theorem requires some further notation for partition multiplication, division, and subpartitions that are analogous to partition divisors.  For two partitions $\lambda, \delta \in \mathcal{P}$, we define $\delta \lambda$ to be the partition obtained by concatenating the parts of $\lambda$ to $\delta$ (and then reordering by size by notational convention). 
We write $\delta | \lambda$ if all of the parts of $\delta$ appear in the partition $\lambda$ with greater than or equal multiplicities and define the {\it subpartition} $\lambda/ \delta$ as the partition obtained by removing those parts, counting multiplicities, from $\lambda$.  

For two functions $f$ and $g$ on partitions, we define their {\it partition Dirichlet convolution} by
\begin{equation}\label{Dirichlet}
(f * g)(\lambda) := \sum_{\delta | \lambda} f(\delta) g(\lambda/\delta),
\end{equation}
where the sum is taken over the subpartitions $\delta$ of $\lambda$ (including the empty partition $\delta=\emptyset$).

To state our results succinctly, we define an auxiliary series that effectively serves to identify different classes of density results. %\begin{definition}\label{A_ndef}
For $|q|<1$, $a(\lambda)$ a function on $\mathcal P$, and $n\geq 1$, let 
\begin{equation}\label{A_ndef}
{A}_n(q):=a\left((n)\right)+\sum_{\operatorname{sm}(\gamma){\geq}^* n}\left[a(\gamma\cdot (n))-a(\gamma)\right] q^{|\gamma|},
\end{equation}
%denotes $>$ if $a(\gamma)$ vanishes at partitions with any part repeated, and denotes $\geq$ otherwise; 
with $(n)\in \mathcal P$ the partition of $n$ with one part, $\gamma\cdot(n)\in \mathcal P$ the partition formed by adjoining part $n$ to partition $\gamma$, and with ${\geq}^*$ denoting $>$ if $a(\lambda)=0$ at partitions with any part repeated, and denoting $\geq$ otherwise. %\end{definition}

We introduce one final piece of terminology: throughout this paper, the formulas we prove hold true for ``nice'' subsets $\mathcal S\subseteq \mathbb N$ amenable to $q$-series density calculations, which we call {\it $q$-commensurate} subsets (see Section \ref{qdensity} below for the precise definition).

\begin{theorem}\label{Thm1}
Let $\mathcal S\subseteq \mathbb N$ be a $q$-commensurate subset, and $a(\lambda)$ a function on $\mathcal P$ such that $a(\emptyset)=1$. If $\lim_{q\to 1}A_n(q)=1$ when $n\in \mathcal S$, then 
\begin{flalign*}
\lim_{q\to 1}\sum_{\operatorname{sm}(\lambda)\in \mathcal S}(\mu_{\mathcal P}*a)(\lambda)q^{|\lambda|}  \  =\  0.\end{flalign*}
If $\lim_{q\to 1}A_n(q)=1+f(n)$ for $f(n)$ an arithmetic function nonzero when $n\in \mathcal S$, then 
\begin{flalign*}
\lim_{q\to 1}\sum_{\operatorname{sm}(\lambda)\in \mathcal S}\frac{(\mu_{\mathcal P}*a)(\lambda)}{f\left(\operatorname{sm}(\lambda)\right)}q^{|\lambda|}  \  =\  d_{\mathcal S}.\end{flalign*}
\end{theorem}
%
%\begin{remark}
%The condition $|A_n(q)|=o(1)$ is sufficient, but we have not proved it is necessary; other analytic conditions may be possible.
%\end{remark}

If we set $f(n)=n$, Theorem \ref{Thm1} gives the partition analogue of \eqref{WangSum}. Immediately we have a nice generalization of formula \eqref{Fdens} with $\zeta=1$.

\begin{corollary}\label{Cor2}
For $\mathcal S\subseteq \mathbb N$ a $q$-commensurate subset, if  $\lim_{q\to 1}A_n(q)=0$ when $n\in \mathcal S$, then  %$\sum_{\lambda\in \mathcal P}a(\lambda)q^{|\lambda|}={o} \left(\sum_{\lambda\in \mathcal P}q^{|\lambda|}\right)$ 
%as $q\to 1$ that 
\begin{flalign*}
-\lim_{q\to 1}\sum_{\operatorname{sm}(\lambda)\in \mathcal S}(\mu_{\mathcal P}*a)(\lambda)q^{|\lambda|}  \  =\  d_{\mathcal S}.\end{flalign*}
\end{corollary}

\begin{remark}
Set $a(\lambda)=\sum_{\delta | \lambda}\mu_{\mathcal P}(\delta)=1$ if $\lambda=\emptyset$ and $=0$ otherwise; one computes $A_n(q)=0$ for all $n\geq 1$. Then the theorem reduces to the main result of \cite{Paper1}, viz. $\lim_{q\to 1}F_{\mathcal S}(q)= d_{\mathcal S}$. %; one can see this immediately from the left side of Theorem \ref{Thm1} as well.
\end{remark}

Furthermore, for an {\it arithmetic} function $a(n)$, Wang defines $b(n) := \sum_{d|n} (\mu * a)(d) \frac{d}{\varphi(d)}$ with $\varphi(k)$ the classical Euler phi function, and in equation (36) of \cite{Wang3} uses M\"{o}bius inversion to show
\begin{equation}\label{W2}
\frac{(\mu * a)(n)}{\varphi(n)} = \frac{(\mu * b)(n)}{n}.
\end{equation}
This along with the main theorem \eqref{WangSum} allows Wang to prove beautiful formulas like
\begin{equation}\label{Wang}
-\lim_{N \to \infty} \sum_{\substack{2 \leq n \leq N \\ p_{\rm{min}}(n) \in \mathcal{S}}} \frac{(\mu * a)(n)}{\varphi(n)} = d_{\mathcal{S}}.
\end{equation}

Replacing $a,b$ with functions on partitions, Theorem \ref{Thm1} yields an analogue of these equations.

\begin{corollary}\label{Cor}
Let $\mathcal S\subseteq \mathbb N$ be a $q$-commensurate subset, $a(\lambda)$ a function on $\mathcal P$ such that $a(\emptyset)=1$, $f$ and $g$ functions on $\mathbb Z_{\geq 0}$ such that $f(n)\neq 0$, $g(n)$ nonzero when $n\in \mathcal S$, and
\begin{flalign*}
b(\lambda)=b_{a,f,g}(\lambda):=\sum_{\delta | \lambda} (\mu_{\mathcal P}*a)(\delta)\frac{g\left(\operatorname{sm}(\delta)\right)}{f\left(\operatorname{sm}(\delta)\right)}.
\end{flalign*}
% b_{a,f,g}(\lambda)$ as in \eqref{b}. 
If $\lim_{q\to 1}A_n(q)=1+f(n)$ for $n\in \mathcal S$, then
\begin{flalign*}
\lim_{q\to 1}\sum_{\operatorname{sm}(\lambda)\in \mathcal S}\frac{(\mu_{\mathcal P}*b)(\lambda)}{g\left(\operatorname{sm}(\lambda)\right)}q^{|\lambda|}  \  =\  d_{\mathcal S}.\end{flalign*}
\end{corollary}

A version of Corollary \ref{Cor} gives an analogue to Wang's identity \eqref{Wang}.

\begin{corollary}\label{Cor5}
Let $\mathcal S\subseteq \mathbb N$ be a $q$-commensurate subset,  and take $a(\lambda), b(\lambda)=b_{a,f,g}(\lambda)$ as in Corollary \ref{Cor}. In analogy with \eqref{A_ndef}, define
\begin{equation*}
{B}_n(q):=b\left((n)\right)+\sum_{\operatorname{sm}(\gamma){\geq}^* n}\left[b(\gamma\cdot (n))-b(\gamma)\right] q^{|\gamma|}.
\end{equation*}
Let $f(n)=-\varphi(n), g(n)=-1,$ so that $b=b_{a,-\varphi, -1}$. If $\lim_{q\to 1}B_n(q)=0$ for $n\in\mathcal S$, then
\begin{flalign*}
-\lim_{q\to 1}\sum_{\operatorname{sm}(\lambda)\in \mathcal S}\frac{(\mu_{\mathcal P}*a)(\lambda)}{\varphi\left(\operatorname{sm}(\lambda)\right)}q^{|\lambda|}  \  =\  d_{\mathcal S},\end{flalign*}
where $\varphi(n)$ is the classical phi function.
\end{corollary}

%
%\begin{corollary}\label{Cor4}
%Let $\mathcal S\subseteq \mathbb N$ be a $q$-commensurate subset, $a(\lambda)$ a function on $\mathcal P$ such that $a(\emptyset)=1$, $f$ a function on $\mathbb Z_{\geq 0}$ such that $f(n)\neq 0$, $\phi(\lambda)$ a function on partitions such that $\sum_{\lambda\in \mathcal P}q^{\phi(\lambda)}$ converges for $|q|<1$, and
%\begin{flalign*}
%\widehat{b}_{f,\phi}(\lambda):=\sum_{\delta | \lambda} \frac{(\mu_{\mathcal P}*a)(\delta)}{f\left(\operatorname{sm}(\delta)\right)}q^{|\delta|-\phi(\delta)}.
%\end{flalign*}
%% b_{a,f,g}(\lambda)$ as in \eqref{b}. 
%If $\lim_{q\to 1}A_n(q)=1+f(n)$, then
%\begin{flalign*}
%\lim_{q\to 1}\sum_{\operatorname{sm}(\lambda)\in \mathcal S}(\mu_{\mathcal P}*\widehat{b}_{f,\phi})(\lambda)q^{\phi(\lambda)}  \  =\  d_{\mathcal S}.\end{flalign*}
%\end{corollary}
%
%

%Let $g(k)=-1$ identically, $f(k)=\varphi(k)$ in Corollary \ref{Cor}.

We prove these limiting formulas in Section \ref{ProofSect}. 

\begin{remark} As noted in Section \ref{qdensity} below, by a Tauberian theorem of Frobenius \cite{Frobenius}, the formulas in this paper hold for {\it every} subset $\mathcal S\subseteq \mathbb N$ if one lets $q\to 1^-$ strictly radially.
\end{remark}

\section{Conceptual background}\label{Sect2}

\subsection{Partition analogue of multiplicative number theory}%Nuts and bolts about the arithmetic partition analogies}

Before proceeding to prove the propositions in Section \ref{Sect1}, let us discuss the multiplicative-additive confluence that provides the context in which results like these --- displaying a fusion of partition-theoretic objects with familiar-looking forms from classical multiplicative number theory --- arise naturally.

In the preceding section we introduced the multiplicative partition norm $N(\lambda)$, the partition product (concatenation), partition division (deleting parts) and subpartition sums that act like sums over divisors, as well as partition analogues of classical arithmetic functions and Dirichlet convolution; this multiplicative theory of additive partitions was introduced in \cite{Robert_zeta, Schneider_arithmetic, SchneiderPhD} and subsequent work \cite{OnoRolenS, Paper1, SchneiderSills}. Some of these analogies, pertinent to this paper, are captured in Table \ref{T1}.

These partition-theoretic structures behave almost identically to their classical counterparts. For instance, one can prove familiar-looking ``divisor sum'' identities like  
\begin{flalign*}
\sum_{\delta | \lambda}\mu_{\mathcal P}(\delta)&= \begin{cases} 1 & \text{if}\  \lambda = \emptyset, \\ 0 & \text{otherwise},
\end{cases}\\
\sum_{\delta | \lambda}\varphi_{\mathcal P}(\delta)&=N(\lambda);
\end{flalign*}
the partition M\"{o}bius inversion formula
\begin{equation}\label{Pmobius}
A(\lambda)=\sum_{\delta | \lambda}a(\delta)\  \   \Longleftrightarrow \  \  
a(\lambda)=\sum_{\delta | \lambda}A(\delta)\mu_{\mathcal P}(\lambda / \delta),
\end{equation}
and closely connected order-of-summation-swapping principle
\begin{equation*}
\sum_{\lambda\in \mathcal P} a(\lambda) \sum_{\delta | \lambda}b(\delta)\  \   =\  \  \sum_{\lambda \in \mathcal P}b(\lambda)\sum_{\gamma \in \mathcal P}a(\gamma \lambda);
\end{equation*}
and other analogues of classical identities, like a {\it partition Cauchy product} formula for absolutely convergent series
\begin{equation}\label{Cauchy}
\left( \sum_{\lambda\in\mathcal P}a(\lambda) \right) \left( \sum_{\lambda\in \mathcal P} b(\lambda) \right) =\sum_{\lambda\in \mathcal P} \sum_{\delta | \lambda} a(\delta) b(\lambda/\delta),%=\sum_{\lambda\in \mathcal P}(f*g)(\lambda),
\end{equation}
which is obviously related to partition convolution defined in \eqref{Dirichlet}.\footnote{We note that the partition convolution operation was initially suggested to the second author by O. Beckwith (personal communication, 2018).}    
We note two useful consequences of \eqref{Cauchy} from \cite{Schneider_arithmetic}, setting $b=\mu_{\mathcal P}$:
\begin{flalign*}
(q;q)_{\infty}^{-1} \sum_{\lambda\in\mathcal P} a(\lambda) q^{|\lambda|} &=\sum_{\lambda\in \mathcal P} q^{|\lambda|}\sum_{\delta | \lambda} a(\delta),\\  \  \  \ 
(q;q)_{\infty} \sum_{\lambda\in\mathcal P} a(\lambda) q^{|\lambda|} &=\sum_{\lambda\in \mathcal P} q^{|\lambda|}\sum_{\delta | \lambda} a(\delta) \mu_{\mathcal P}(\lambda/\delta).
\end{flalign*}
These are $q$-series analogues of the classical Dirichlet series identities (see Table \ref{T1})
\begin{flalign*}
\zeta(s)^{-1} \sum_{n=1}^{\infty} f(n) n^{-s} &=\sum_{n=1}^{\infty} n^{-s}\sum_{d | n} f(d),\\  \  \  \ 
\zeta(s) \sum_{n=1}^{\infty} f(n) n^{-s} &=\sum_{n=1}^{\infty} n^{-s}\sum_{d | n} f(d)\mu(n/d).
\end{flalign*}

This multiplicative-additive theory also contains broad classes of {\it partition zeta functions}
\begin{equation*}
\zeta_{\mathcal P'}(s):=\sum_{\lambda \in \mathcal P'}N(\lambda)^{-s} \  \  \  \  (\text{Re}(s)>1)\end{equation*}
where $P'\subsetneq \mathcal P$, that yield beautiful evaluations \cite{Robert_zeta} and interesting analytic information \cite{OnoRolenS, SchneiderSills}. There also exist {\it partition Dirichlet series} \cite{OnoRolenS, SchneiderPhD} that enjoy identities generalizing classical Dirichlet series relations, e.g. if $\mathcal P_{\mathcal S}$ denotes partitions with all parts in any $\mathcal S\subset \mathbb N$,
\begin{flalign*}
\sum_{\lambda\in \mathcal P_{\mathcal S}}\mu_{\mathcal P}(\lambda) N(\lambda)^{-s}&=\frac{1}{\zeta_{\mathcal P_{\mathcal S}}(s)}\  \  \  \  \  \  \   (\text{Re}(s)>1),\\
\sum_{\lambda\in \mathcal P_{\mathcal S}}\varphi_{\mathcal P}(\lambda) N(\lambda)^{-s}&=\frac{\zeta_{\mathcal P_{\mathcal S}}(s-1)}{\zeta_{\mathcal P_{\mathcal S}}(s)}\  \  (\text{Re}(s)>2).
\end{flalign*}
Taking $\mathcal S = \mathbb P$ regains the classical cases. While this is indeed looking a lot like multiplicative number theory, the same superstructure encompasses $q$-hypergeometric series, partition bijections and other aspects of additive number theory, and recovers classical theorems in those areas as well (see \cite{SchneiderPhD} for further reading).
 
% 
%Against this conceptual backdrop, it is natural to look for correspondences between arithmetic densities, which are usually related to multiplicative patterns. % like square-freeness, arithmetic progressions, average values of arithmetic functions, etc., and the densities (so to speak) of analogous subsets of partitions.
%

\section{Computing arithmetic densities with $q$-series}\label{qdensity}

\subsection{The $q$-density statistic}

For $|q|<1$, we recall that the function %In \cite{Paper1}, the present authors define a density function 
$F_{\mathcal S}(q)$ from \eqref{OnePointNine} has two natural representations, one involving $\operatorname{sm}(\lambda)$ and one  involving $\operatorname{lg}(\lambda)$, the smallest and largest parts of partition $\lambda$, respectively. Namely, we have %in the form%which can be written as a quotient of $q$-series indexed by partitions
\begin{flalign}\label{F_S}
%\lim_{N\to \infty}\frac{1}{\log N}{\sum_{n \in \mathcal S(N)}\frac{1}{n}}=
F_{\mathcal S}(q):=\sum_{\operatorname{sm}(\lambda)\in \mathcal S}\mu_{\mathcal P}^*(\lambda)q^{|\lambda|}=\frac{\sum_{\text{lg}(\lambda)\in \mathcal S}q^{|\lambda|} }{\sum_{\lambda \in \mathcal P}q^{|\lambda|}}%\\&=(q;q)_{\infty}\sum_{\text{lg}(\lambda)\in \mathcal S}q^{|\lambda|}=
=(q;q)_{\infty}\sum_{n\in \mathcal S}\frac{q^n}{(q;q)_n},
\end{flalign}
noting $\text{sm}(\emptyset)= \text{lg}(\emptyset):=0$. 
We proved in \cite{Paper1} that $\lim_{q\to 1}F_{\mathcal S}(q)=d_{\mathcal S}$ for special subsets $\mathcal S\subseteq \mathbb N$. %Since a nonzero arithmetic density $d_{\mathcal S}$ is defined by an indeterminate limit of form $\frac{\infty}{\infty}$, it is natural to exploit ratios of divergent series using tools from analysis to compute this limiting value. Moreover, it is conceptually satisfying to compute densities as limiting values of quotients of {\it infinite} series, instead of using partial sums; for a partial sum only presents to one's imagination the finite beginning of the set $\mathcal S$, while the arithmetic density itself is a phenomenon manifesting entirely in the terms at the (deleted) tail of the series. % Due to the great flexibility of $q$-hypergeometric series with respect to reparameterizations and transformations (e.g. see \cite{Andrews_monograph, Fine, Ono_web}), along with their well-understood behaviors as $|q|\to 1$, one imagines they might present further formulas for $d_{\mathcal S}$ like the ones proved in \cite{Paper1}. 

The story arc represented in the works of Alladi \cite{Alladi1}, Dawsey \cite{Dawsey}, Wang \cite{Wang1, Wang2, Wang3}, et al., together with our own work in \cite{Paper1}, stems from the observation that \eqref{First} can be interpreted as a statement about arithmetic density. Now, this statement itself can be interpreted via L'Hospital's rule and Lambert series as a statement about $q$-series, since
\begin{equation*}
%\lim_{N\to \infty}\frac{1}{\log N}{\sum_{n \in \mathcal S(N)}\frac{1}{n}}=
\sum_{n\geq 2}\frac{\mu(n)}{n}\  =\  \lim_{q\to 1}(1-q)\left(\sum_{n\geq 1}\frac{\mu(n)q^n}{1-q^n}-\frac{\mu(1)q}{1-q}\right)\  =\  \lim_{q\to 1}(1-q)\left(q-\frac{q}{1-q}\right)\  =\  -1,
\end{equation*}
assuming the implicit order-of-limits swap with ``$\lim_{q\to 1}$'' and the infinite series is valid. 

The classical $q$-binomial theorem was at the heart of our results in \cite{Paper1}. In order to prove the main results of this paper involving less familiar-looking forms, we sketch a theory of $q$-series computations of arithmetic density from first principles, % such as we exploit in \cite{Paper1} and this work, % such as one sees in \cite{Paper1} and the present work. %, which feels distinct in flavor from the usual Dirichlet series limits at the pole $s=1$. 
based on a simple statistic.% we call the {\it $q$-density} of $\mathcal S$.

\begin{definition} For $|q|<1$ we define the {\it $q$-density} of $\mathcal  S \subseteq \mathbb N$ by the quotient
\begin{equation}\label{qdens}
d_{\mathcal S}(q):=\frac{\sum_{n\in \mathcal S}q^n}{\sum_{n\geq 0}q^n}=(1-q)\sum_{n\in \mathcal S}q^n.
\end{equation}
\end{definition}

The $q$-density is a geometric series analogue of \eqref{F_S}.\footnote{For $q\in [0,1)$, equation \eqref{qdens} can be interpreted as the expected value of the indicator function for the subset $\mathcal S$, given the probability function $P(n)=q^{n-1}(1-q)$ %(which we don't interpret here) 
for random $n\in \mathbb N$.} %, giving a probability distribution.}%. Na\"{i}vely, one guesses this should equal $d_{\mathcal S}$.}
%Na\"{i}vely, 
Of course, from both the shape of $d_{\mathcal S}(q)$ and known instances of $F_{\mathcal S}(q)$, %at least for certain ``good'' subsets $\mathcal S$, 
one wants to see the limiting behavior% (noting the numerator above is $\leq N$).}
  %(at least for ``good'' subsets $\mathcal S$) %anticipates that  %One would like to see that %, as $q\to 1$, the intuitive limit
\begin{equation}\label{qdream}
\lim_{q\to 1}  \frac{\sum_{n\in \mathcal S} q^n}{\sum_{n\geq 0}q^n}=\lim_{N\to \infty}\frac{\sum_{n\in \mathcal S(N)}1}{\sum_{0\leq n <N}1}=:d_{\mathcal S},
\end{equation}
where $\mathcal S(N)=\{n\in \mathcal S : n \leq N \}$. %It is believable that $d_{\mathcal S}(q)\to d_{\mathcal S}$ as $q\to 1$, at least for ``nice'' subsets of $\mathbb N$, if one lets $q$ approach $1$ through the sequence $q=\frac{1}{2}, \frac{2}{3}, \frac{3}{4}, \dots, 1-\frac{1}{N},\dots$. The left-hand limit of \eqref{qdream} then becomes
%$%\begin{equation*}\label{ratio}
%\lim_{N\to \infty}  \frac{1}{N}\sum_{n\in \mathcal S} (1-\frac{1}{N})^n,
%$ %%\end{equation*}
%which looks even more like the right-hand limit.
 % analytic in the unit disk. 
%while the right-hand limit is $%\begin{equation*}\label{ratio}
%\lim_{q\to 1} (1-q) \sum_{n\in \mathcal S\left(\frac{1}{1-q} \right)}1
%$, %%\end{equation*}
%something of a role reversal of limits. 
It is a consequence of a Tauberian theorem due to Frobenius with a converse by Hardy-Littlewood (see \cite{Frobenius, Tauberian}), that in fact \eqref{qdream} holds for {every} subset $\mathcal S\subseteq \mathbb N$ if $q\to 1$ radially.\footnote{Indeed, \eqref{qdream} holds if $q\to 1$ in a ``Stolz sector'' of the unit circle; for a somewhat more general context, the reader is referred to the second author's study \cite{Abelian}, following up on the present work.} %It is beyond the scope of this paper to pursue this issue in full generality; i
In \cite{Paper1}, we prove arithmetic density formulas using other paths, and approaching other roots of unity; we do not rigorously prove general limiting properties here (certainly such a study would be useful), but side-step the issue somewhat circularly by defining it to be a {\it property} of a subset of $\mathbb N$ %that might include all subsets, 
that the limit computations we undertake are valid. % in some generality. % in {some} region. %with arbitrary generality. %we conjecture all subsets of $\mathbb N$ possess.  % (without proving whether subsets have this property). %which may depend on features of subset $\mathcal S$.%\footnote{It is beyond the scope of this paper to prove further convergence conditions; see the second author's follow-up paper \cite{Abelian}.} %and observe it appears to depend on $S$,

\begin{definition}
We define a subset $\mathcal  S \subseteq \mathbb N$ to be {\it $q$-commensurate} %commensurate, faithful, regular, commensurate, parallel, corresponding, normal, fair, honest, commensurate, unbiased, just, rightful, matching, sonorous, symmetric, sympathetic, suitable, consonant , accordant, balanced, coincident, comparable, corresponding
if $\lim_{q\to 1}d_{\mathcal S}(q)=d_{\mathcal S}$ is approached sufficiently rapidly for the asymptotic proofs we employ here to hold.%\footnote{In full generality, a subset $\mathcal S$ being $q$-commensurate is equivalent to one's being able, for that subset, to interchange order of limits as in \eqref{qdream}, %\footnote{Swapping limits is highly nontrivial; G. H. Hardy called it ``one of the most important [problems] in mathematics.''\cite{Hardy}.}, 
%{\it and} to having all implicit error terms in the following asymptotic equalities of lower order than the main terms, so the errors vanish when multiplied by $(q;q)_{\infty}$  as $q\to 1$.}% for $\zeta$ a root of unity.% If this limit holds only when $q\to \zeta$ radially, we say $\mathcal S$ is ``radially $q$-commensurate%\footnote{One might say $\mathcal S$ is ``$q$-commensurate'' if this is the case.}.
\end{definition}

\begin{remark}
In full generality, a subset $\mathcal S$ being $q$-commensurate is equivalent to one's being able, for that subset, to interchange order of limits in the calculation
%\begin{equation}\label{ratio2}
$d_{\mathcal S}: =\lim_{N\to \infty} \lim_{q\to 1} \frac{\sum_{n\in \mathcal S(N)} q^n}{\sum_{0 \leq n < N} q^n}
$, %\footnote{Swapping limits is highly nontrivial; G. H. Hardy called it ``one of the most important [problems] in mathematics.''\cite{Hardy}.}, 
{\it and} to having all implicit error terms in the following asymptotic equalities of lower order than the main terms, so the errors vanish when multiplied by $(q;q)_{\infty}$  as $q\to 1$.\end{remark}

Immediately we see that $\mathcal S=\mathbb N$ is $q$-commensurate since $d_{\mathbb N}(q)=q\to 1=d_{\mathbb N}$.  %On the other hand, we have not found examples of {\it non}-$q$-commensurate $\mathcal S$, nor proved their existence
 %.} 
%x\end{equation}
%\end{remark}
%If one thinks of the set of values $\{q^n(1-q) : n\geq 0\}$ as a probability distribution, then $ 
%It seems likely that many (if not all) subsets of $\mathbb N$ are $q$-commensurate, although we have not proved this.
The analogy between $d_{\mathcal S}(q)$ and $F_{\mathcal S}(q)$ holds quite closely with respect to density. %, one wants 
%to see
%\begin{equation}
%\lim_{q\to 1} d_{\mathcal S}(q)=d_{\mathcal S}, %\lim_{N\to \infty}\frac{\sum_{n\in \mathcal S(N)}1}{\sum_{1\leq n \leq N}1},
%\end{equation}
%since, na\"{i}vely speaking, the numerators and denominators approach equality. 
\begin{theorem}\label{qdensthm}
%Given sufficient analytic conditions dependent on the set $\mathcal  S \subseteq \mathbb N$, w
Let $r\in \mathbb Z,t \in \mathbb N$. For subsets $\mathcal S_{r,t}:=\{n\geq 0 : n \equiv r \   (\operatorname{mod}t)\}$ we have that  $$\lim_{q\to 1}d_{\mathcal S_{r,t}}(q)=d_{\mathcal S_{r,t}}=\frac{1}{t}.$$
\end{theorem}
%\vfill
%  by the Davenport-Erd\H{o}s theorem \cite{DE}.
%
%
%
%
%Here we explore the use of poles of $q$-series to compute arithmetic densities a little more generally. The arithmetic density $d_{\mathcal S}$ of $\mathcal  S \subseteq \mathbb N$ represents an indeterminate limit, viz.
%\begin{equation}\label{arith}
%d_{\mathcal S}=\lim_{N\to \infty}\frac{\#S(N)}{N}=\lim_{N\to \infty}\frac{\sum_{n\in \mathcal S(N)}1}{\sum_{1\leq n \leq N}1}
%\end{equation}
%with $\mathcal S(N):=\{k \in \mathcal S : k \leq N\}$; other indeterminate limits such as the ratios of Dirichlet series as $s\to 1^+$, and quotients of $q$-series as $q$ approaches roots of unity as in the present study, give arithmetic densities as well -- somewhat astoundingly, like infinite series versions of L'Hospital's rule. 
%We wish to study further examples of $q$-series density calculations; to aid us in this study, we will make use of two auxiliary statistics. To begin, let us 
\begin{proof}
%We have left the analytic conditions suitably vague due to the heuristic aim of this section; that level of analytic detail is beyond the scope of this short paper. We note that in the ``base case'' of $S=\mathbb N$, the commensurate density $d_{\mathbb N}=1$ is obtained, if one questions whether such conditions are even possible. 
The claimed limit %holds %at least for sets $\mathcal S$ of the form $\mathcal S_{r,t}:=\{n\geq 0 : n \equiv r \   (\operatorname{mod}t)\}$, since 
follows from geometric series and L'Hospital's rule:
\begin{equation*}
\lim_{q\to 1}d_{\mathcal S_{r,t}}(q)\  = \lim_{q\to 1}(1-q)\sum_{k\geq 0}q^{kt+r}\  = \lim_{q\to 1}\frac{q^r(1-q)}{1-q^t}\  =\  \frac{1}{t}.\end{equation*} \end{proof}

Theorem \ref{qdensthm} is a geometric series version of Theorem 1.3 in \cite{Paper1}. There is also a $q$-density analogue of Corollary 1.5 in \cite{Paper1}. %, as follows.

\begin{corollary}\label{qdensthm2}
%Given sufficient analytic conditions dependent on the set $\mathcal  S \subseteq \mathbb N$, w
For subsets $\mathcal S_{\operatorname{fr}}^{(k)}\subset \mathbb N$ of $k$th power-free integers, $k\geq 2$, we have that $$\lim_{q\to 1}d_{\mathcal S_{\operatorname{fr}}^{(k)}}(q)=d_{\mathcal S_{\operatorname{fr}}^{(k)}}=\frac{1}{\zeta(k)}.$$
\end{corollary}

\begin{proof}
The steps of the proof of Cor. 1.5 in \cite{Paper1} hold identically here, only one replaces $F_{\mathcal S_{r,M}}(q)$ with $d_{\mathcal S_{r,M}}(q)$, then uses Theorem \ref{qdensthm} in place of Th. 1.3 of \cite{Paper1}.  \end{proof}

%, as well as that analogous computations to those for $F_{S}(q)$ will be hold. 
%In such cases, t
%\begin{remark} 
%\begin{remark}
Thus the subsets $\mathcal S_{r,t}$ and $\mathcal S_{\operatorname{fr}}^{(k)}$ are $q$-commensurate. %\end{remark}
Table 2 
%\ref{T2} 
illustrates the $k=2$ (square-free) case of Corollary \ref{qdensthm2} as $q\to 1$ radially.\footnote{The values in Table 2 were computed by A. V. Sills in Mathematica.} 

%
%\begin{remark}
%The formulas we prove in this paper hold universally for any subset $\mathcal S\subseteq \mathbb N$ if one takes $q\to 1$ radially, by the result of Frobenius mentioned above.
%\end{remark}

%\end{remark}
%Having seen now that $q$-commensurate subsets $\mathcal S$ exist --- particularly in the above cases of subsets central to the study of densities --- and that in each case $d_{\mathcal S}$ can be computed from two different $q$-series $d_{\mathcal S}(q)$ and $F_{\mathcal S}(q)$, it is natural to wonder: are there other such $q$-series limit identities for arithmetic densities?

%
%
%\begin{example}
%Here we approximate the density of $\mathcal S_{\text{fr}}^{(4)}$, the fourth power-free positive integers. Since $\zeta(4)=\pi^4/90$, it follows that the arithmetic density of $\mathcal S_{\text{fr}}^{(4)}$ is
%$\frac{90}{\pi^4} \approx 0.923938...$. Here we choose $N=5$ and compute the arithmetic density of $\mathcal S_{\text{fr}}^{(4)}(5)$, the positive integers which are not divisible by $2^4, 3^4$, and $5^4$. The density
%of this set is $208/225 \approx 0.924444...$. This density is fairly close to the density of fourth power-free integers because the convergence in the $N$ aspect is significantly faster for fourth power-free integers than for square-free integers as discussed above.
\ 
 
\begingroup\label{T2}
\begin{center}
\begin{tabular}{ | c | c | }%\label{T2}
\hline
$q$ & $d_{\mathcal S_{\text{fr}}^{(2)}}(q)$\\ \hline
$0.90$ & $0.739971...$ \\ \hline
$0.91$ & $0.729108...$ \\ \hline
$0.92$ & $0.717828...$ \\ \hline
$0.93$ & $0.706100...$ \\ \hline
$0.94$ & $0.693895...$ \\ \hline
$0.95$ & $0.681180...$ \\ \hline
$0.96$ & $0.667921...$ \\ \hline
$0.97$ & $0.654062...$ \\ \hline
$0.98$ & $0.639523...$ \\ \hline
$0.99$ & $0.624215...$ \\ \hline
$0.999$ & $0.609613...$ \\ \hline
$0.9999$ & $0.608069...$ \\  
\hline
\end{tabular}
\smallskip
\captionof{table}{\textit{Illustration of $d_{\mathcal S_{\text{fr}}^{(2)}}(q)$ approaching $\frac{6}{\pi^2}=0.607927...$}}
\end{center}
%\vspace{0.5cm}

\subsection{Further $q$-series density formulas}

The $q$-density statistic can be used as a building block for further arithmetic density relations.\footnote{E.g. see Section \ref{ExSect} below.} 
%Immediately, if $\mathcal  S \subseteq \mathbb N$ has commensurate $q$-density then L'Hospital's rule implies for every $k\geq 1$ that
%\begin{equation}
%\lim_{q\to 1}\frac{\frac{\text{d}^k}{\text{d}q^k}\sum_{n\in \mathcal S}q^n}{\frac{\text{d}^k}{\text{d}q^k}\sum_{n\geq 0}q^n}\  =\  d_{\mathcal S}.
%\end{equation}
%Taking antiderivatives of the numerator and denominator of $\lim_{q\to 1}\frac{\sum_{n\in \mathcal S}q^n}{\sum_{n\geq 0}q^n}$ (constants of integration $C=0$), then applying L'Hospital's rule, yields another density relation:
% \begin{flalign}
%-\lim_{q\to 1} \frac{1}{\operatorname{log}(1-q)}\sum_{n\in \mathcal S}\frac{q^n}{n}\  =\   d_{\mathcal S},
%\end{flalign}
% where $\operatorname{log} z$ denotes the principle branch of the complex logarithm. 
 %\end{theorem}
%\begin{proof}
%Apply L'Hospital's rule on the left-hand side; then the limit is given by \eqref{qdream}.
%\end{proof}
%We will assume all sets $\mathcal  S \subseteq \mathbb N$ below to be such that $\lim_{q\to 1}d_{\mathcal S}(q) = d_{\mathcal S}$, like those above. 
For $ \mathcal S \subseteq \mathbb N$ a $q$-commensurate subset\footnote{Or any subset of $\mathbb N$ if $q\to 1$ radially, by the result of Frobenius cited above.}, as $q \to 1$ %since $d_{\mathcal S}(q)\cdot q \sim d_{\mathcal S}$ as $q\to 1$ it follows from 
it follows from \eqref{qdream} that %, we have%\footnote{In fact, the asymptotic holds if one replaces $q$ with $q^{f}, f \geq 0,$ in the numerator on the right.}
\begin{equation}\label{qasymp}
{\sum_{n\in \mathcal S}q^n} \  \sim\  d_{\mathcal S} \cdot \frac{q}{1-q}.
\end{equation}

%More general expressions can be obtained. 
Because geometric series are central components of partition generating functions and $q$-series, the $q$-density may be embedded in diverse partition relations. For example, one can prove the central claims of \cite{Paper1} (with a sufficient condition) by applying generating function methods to \eqref{qasymp}. 
%

%Theorem \ref{muthm} feels like an analogue of 
%\begin{equation}
%\lim_{q\to 1}\sum_{\text{sm}(\lambda)\in \mathcal S}\mu_{\mathcal P}^*(\lambda)q^{|\lambda|}=d_{\mathcal S}.
%\end{equation}

\begin{theorem}\label{qgeneral}
For $ \mathcal S \subseteq \mathbb N$ a $q$-commensurate subset, we have 
\begin{equation*}
\lim_{q\to 1} F_{\mathcal S}(q)\  =\  -\lim_{q\to 1}\  \sum_{n\in \mathcal S}\sum_{k\geq 1}\frac{(-1)^k q^{nk+\frac{k(k-1)}{2}}}{(q;q)_{k-1}}\  = \  d_{\mathcal S}.
\end{equation*}
\end{theorem}

\begin{proof}
Take $q\mapsto q^k$ in \eqref{qasymp}.  Multiply through by $(-1)^{k} q^{\frac{k(k-1)}{2}}(q;q)_{k-1}^{-1}$ and sum both sides over $k\geq 1$, swapping order of summation on the left to give
\begin{equation}\label{qleft}
\sum_{n \in \mathcal S} \sum_{k\geq 1} \frac{(-1)^{k} q^{nk+\frac{k(k-1)}{2}}}{(q;q)_{k-1}}.
\end{equation}
For each $k\geq 1$, the factor $(q;q)_{k-1}^{-1}$ generates partitions with largest part strictly $<k$. The factor $q^{nk}$ adjoins a largest part $k$ with multiplicity $n\in \mathcal S$ to each partition. The $q^{k(k-1)/2}=q^{1+2+3+...+(k-1)}$ factor guarantees at least one part of each size $<k$. Thus \eqref{qleft} is the generating function for partitions $\gamma$ with largest part having multiplicity in $\mathcal S$, and with every natural number $<\operatorname{lg}(\gamma)$ appearing as a part, weighted by $(-1)^{\operatorname{lg}(\gamma)}=(-1)^k$. Under conjugation, %consideration of Young diagrams proves 
this set of partitions $\gamma$ maps to partitions $\lambda$ into distinct parts with smallest part  $\operatorname{sm}(\lambda)\in \mathcal S$, weighted by $\mu_{\mathcal P}(\lambda)=(-1)^{\ell(\lambda)}=(-1)^k$ which is nonzero since $\lambda$ has no repeated part. Multiplying by $-1$ gives $\mu_{\mathcal P}^*(\lambda):=-\mu_{\mathcal P}(\lambda)$; thus
\begin{equation*}
-\sum_{n\in \mathcal S}\sum_{k\geq 1}\frac{(-1)^k q^{nk+\frac{k(k-1)}{2}}}{(q;q)_{k-1}}=\sum_{\operatorname{sm}(\lambda) \in \mathcal S} \mu_{\mathcal P}^*(\lambda) q^{|\lambda|}= F_{\mathcal S}(q).
\end{equation*}
Manipulating the right-hand side of \eqref{qasymp} accordingly, we then have
\begin{equation*}
-\sum_{n\in \mathcal S}\sum_{k\geq 1}\frac{(-1)^k q^{nk+\frac{k(k-1)}{2}}}{(q;q)_{k-1}}\  \sim \   -d_{\mathcal S}\cdot \sum_{k\geq 1}\frac{(-1)^k q^{\frac{k(k+1)}{2}}}{(q;q)_{k}}=-d_{\mathcal S}\cdot \left((q;q)_{\infty}-1 \right),\end{equation*}
using an identity of Euler in the last step, which clearly approaches $d_{\mathcal S}$ as $q\to 1$. \end{proof}

We don't necessarily need to restrict {\it smallest} parts to $\mathcal S$ to give formulas for $d_{\mathcal S}$. For instance, by similar methods, an analogous formula extends to the ``largest parts'' case. %

\begin{theorem} \label{qgeneral2}
%Let $\operatorname{lg}(\lambda)$ denote the largest part of partition $\lambda$.  
For $ \mathcal S \subseteq \mathbb N$ a $q$-commensurate subset, we have 
\begin{equation*}
\lim_{q\to 1}\sum_{\operatorname{lg}(\lambda)\in \mathcal S}\mu_{\mathcal P}^*(\lambda)q^{|\lambda|}=d_{\mathcal S}.
\end{equation*}
\end{theorem}

\begin{proof}
Take $q\mapsto q^k$ in \eqref{qasymp}. Multiply both sides by $(q;q)_{k-1}^{-1}$, sum over $k\geq 1$, then swap order of summation and make the change of indices $k\mapsto k+1$ on the left side, to give %and take $k\mapsto k+1$ in the indices on the left to give
\begin{flalign}\label{asymp3}
\sum_{n\in \mathcal S}q^n\sum_{k\geq 0}\frac{q^{nk}}{(q;q)_{k}}&\  \sim\   d_{\mathcal S}\cdot \sum_{k\geq 1}\frac{q^k}{(q;q)_k}.
\end{flalign}
%The sum on the left generates partitions with largest part having multiplicity in $\mathcal S$ (which conjugates to the set of partitions with smallest part in $\mathcal S$). 
By the $q$-binomial theorem, the inner sum over $k\geq 0$ is equal to $(q^n;q)_{\infty}^{-1}$, and the summation on the right equals $(q;q)_{\infty}^{-1}-1$. Then multiplying both sides of \eqref{asymp3} by $(q;q)_{\infty}$ implies by well-known generating function arguments that, as $q\to 1$, we have
\begin{equation}\label{asymp4}
\sum_{n\in \mathcal S}q^n(q;q)_{n-1}=\sum_{\operatorname{lg}(\lambda)\in \mathcal S}\mu_{\mathcal P}^*(\lambda)q^{|\lambda|}\  \sim\  d_{\mathcal S}.
\end{equation}
%as $q\to 1$. 
\end{proof}

\begin{remark}
Taken together, Theorems \ref{qgeneral} and \ref{qgeneral2} present an ``Alladi duality'' connecting relations over minimal and maximal elements of multisets, as exemplified in \cite{Alladi1, Alladi_lecture, Wang3}. Whereas in \cite{Alladi1} the smallest and largest prime factors of $n$ induce a duality, here the duality is between the smallest and largest parts of $\lambda$. \end{remark}

\section{Proof of Theorem \ref{Thm1} and its corollaries}\label{ProofSect}
\subsection{Partition convolution identities}
%\section{Proof of Theorem 1.1 and other identities}
Now we turn our attention %to formulas involving partition-theoretic Dirichlet convolution as in \eqref{Dirichlet} and \eqref{Cauchy}, with an eye 
toward proving the partition convolution identities analogous to formulas of Wang \cite{Wang3}. It is an instance of \eqref{Cauchy} for arbitrary partition-theoretic functions $a(\lambda), b(\lambda)$ that
\begin{equation}\label{Cauchy2}
\left( \sum_{\lambda\in\mathcal P}a(\lambda) q^{|\lambda|}\right) \left( \sum_{\lambda\in \mathcal P} b(\lambda)q^{|\lambda|} \right)=\sum_{\lambda\in \mathcal P}(a*b)(\lambda)q^{|\lambda|} =\sum_{n\geq 0}\sum_{\operatorname{sm}(\lambda)=n}(a*b)(\lambda)q^{|\lambda|},
\end{equation}
where we now take $*$ to be partition-theoretic convolution (which is apparent from the context). We record a useful formula for the inner sum $\sum_{\operatorname{sm}(\lambda)=n}$ of the right-hand double series. This will serve as another building block for further arithmetic density formulas.

\begin{lemma}\label{convolemma}% \label{qgeneral2}
%Let $\operatorname{lg}(\lambda)$ denote the largest part of partition $\lambda$.  
For $a(\lambda), b(\lambda)$ functions on $\mathcal P$, we have
\begin{flalign*}
\sum_{\operatorname{sm}(\lambda)=n}(a*b)(\lambda)q^{|\lambda|}&=\left(\sum_{\operatorname{sm}(\gamma)=n}a(\gamma)q^{|\gamma|}\right)\left(b(\emptyset)+\sum_{\operatorname{sm}(\gamma){\geq}^* n}b(\gamma)q^{|\gamma|}\right)\\
& + \left(a(\emptyset)+\sum_{\operatorname{sm}(\gamma){\geq}^* n}a(\gamma)q^{|\gamma|}\right)\left(\sum_{\operatorname{sm}(\gamma)=n}b(\gamma)q^{|\gamma|}\right),\end{flalign*}where ${\geq}^*$ denotes $>$ if the associated summands vanish at partitions with any part repeated, and denotes $\geq$ otherwise.
\end{lemma}

\begin{proof}
    The right-hand side of Lemma \ref{convolemma} generates the partitions $\lambda$ with smallest part $n$ on the left. That the summands on the left should have the form $q^{|\lambda|}\sum_{\delta|\lambda} a(\delta)b(\lambda/\delta) = \sum_{\delta \delta'=\lambda}\left(a(\delta)q^{|\delta|}\cdot b(\delta')q^{|\delta'|}\right)$ is obvious from standard partition generating function ideas, e.g. \eqref{Cauchy2}. Now consider $\lambda=\delta\delta'$ on the left, choosing $\delta$ from a sum with coefficients $a(\gamma)$ on the right-hand side, and $\delta'$ from a sum with coefficients $b(\gamma)$. If subpartition $\delta$ contains the copy of part $n$ satisfying the ``$\operatorname{sm}(\lambda)=n$'' condition, then $\lambda$ is generated by the first term ``$(\sum_{\operatorname{sm}(\gamma)=n})(b+\sum_{\operatorname{sm}(\gamma){\geq}^*n})$'' on the right: $\delta$ comes from the $\sum_{\operatorname{sm}(\gamma)=n}$ factor, and $\delta'$ comes from the $b+\sum_{\operatorname{sm}(\gamma){\geq}^*n}$ factor. If, on the other hand, $\delta'$ contains the designated smallest part $n$, then $\lambda$ is generated by the second term ``$(a+\sum_{\operatorname{sm}(\gamma){\geq}^*n})(\sum_{\operatorname{sm}(\gamma)=n})$'' on the right: $\delta$ comes from $a+\sum_{\operatorname{sm}(\gamma){\geq}^*n}$ and $\delta'$ comes from the $\sum_{\operatorname{sm}(\gamma)=n}$ factor. %This generating function argument is the content of the theorem.  
\end{proof}

%Since partition-theoretic convolution arises from a generalization of classical arithmetic, one anticipates partition analogues of classical convolution identities to arise from Lemma \ref{convolemma} if one chooses $a(\lambda),b(\lambda)$ to be $\mu_{\mathcal P}(\lambda), \varphi_{\mathcal P}(\lambda),$ etc. %Moreover, if either function $a,b$ is of the form $\prod_{\lambda_i\in \lambda}\phi(\lambda_i)$ for some arithmetic function $\phi$, with ``$\lambda_i\in \lambda$'' indicating the product is over the parts $\lambda_i\in \mathbb N$ of $\lambda\in \mathcal P$, then 

We wish to find formulas with coefficients of the form $\mu_{\mathcal P}*a$, to yield Theorem \ref{Thm1}. Now, for $|q|<1$, recall the $q$-bracket $\left<a\right>_q:=\frac{\sum_{\lambda \in \mathcal P}a(\lambda)q^{|\lambda|}}{\sum_{\lambda \in \mathcal P}q^{|\lambda|}}=(q;q)_{\infty}\sum_{\lambda \in \mathcal P}a(\lambda)q^{|\lambda|}$ of Bloch and Okounkov. The $q$-bracket can be interpreted as the expected value of $a(\lambda)$ over all partitions, given certain hypotheses from statistical physics; %given the probability function $P(\lambda)=q^{|\lambda|}(q;q)_{\infty}$ for a random $\lambda\in \mathcal P$. 
it is connected to modular, quasi-modular and $p$-adic modular forms (see \cite{B-O, BOW, Padic, Zagier}). Then \eqref{Cauchy2} gives %, as $(q;q)_{\infty}=\sum_{\lambda\in \mathcal P}\mu_{\mathcal P}(\lambda)q^{|\lambda|}$, that
\begin{equation}
\sum_{\lambda \in \mathcal P}(\mu_{\mathcal P}*a)(\lambda)q^{|\lambda|}=\left<a\right>_q,\end{equation}
% an arithmetic function with $f(0):=1$, we have 
which of course leads to further partition convolution relations from $q$-bracket examples in the literature. Taken over all partitions $\lambda$, we see $\mu_{\mathcal P}*a$ induces a type of averaging phenomenon, as well as a connection to the theory of modular forms. %It is conceivable that if $a$ is an indicator function equal to 1 if $\text{sm}(\lambda)\in \mathcal S$ and equal to 0 otherwise, that the average value of $\text{sm}$ should be related to the arithmetic density of $\mathcal S$. 
%Then in Example \ref{qbracketcor}, partition convolution induces an averaging phenomenon on the right-hand side. 
%In the next example, a widely studied combinatorial structure is generated.

To restrict smallest parts to the subset $\mathcal S$, we require a specialization of Lemma \ref{convolemma}.

\begin{lemma}\label{convolemma2}% \label{qgeneral2}
%Let $\operatorname{lg}(\lambda)$ denote the largest part of partition $\lambda$.  
For $a(\lambda)$ a function on $\mathcal P$, we have
\begin{flalign*}
\sum_{\operatorname{sm}(\lambda)=n}(\mu_{\mathcal P}*a)(\lambda)q^{|\lambda|}  = \widetilde{A}_n(q)\cdot q^n(q^{n+1};q)_{\infty} %\left(a\left((n)\right)-a(\emptyset)+\sum_{\operatorname{sm}(\gamma){\geq}^* n}\left[a(\gamma\cdot (n))-a(\gamma)\right] q^{|\gamma|}\right)
,\end{flalign*}
where 
$\widetilde{A}_n(q):=A_n(q)-a(\emptyset)$ with $A_n(q)$ as in \eqref{A_ndef}.
\end{lemma}

\begin{proof}
Swap $a,b$ in Lemma \ref{convolemma}, noting $a*b=b*a$, and let $b(\lambda)=\mu_{\mathcal P}(\lambda)$ to yield 
\begin{flalign*}
\sum_{\operatorname{sm}(\lambda)=n}(\mu_{\mathcal P}*a)(\lambda)q^{|\lambda|}=&
-q^n(q^{n+1};q)_{\infty}\left(a(\emptyset)+\sum_{\operatorname{sm}(\gamma){\geq}^* n}a(\gamma)q^{|\gamma|}\right)\\ \nonumber &+(q^{n+1};q)_{\infty}\left(\sum_{\operatorname{sm}(\gamma)=n}a(\gamma)q^{|\gamma|}\right).\end{flalign*}
Factor $q^n$ out of the $\sum_{\operatorname{sm}(\gamma)=n}$ factor on the right; the summands undergo the transformation $a\left(\gamma)q^{|\gamma|}\mapsto a(\gamma\cdot(n)\right)q^{|\gamma|}$. Then a little arithmetic gives 
\begin{flalign*}
\sum_{\operatorname{sm}(\lambda)=n}(\mu_{\mathcal P}*a)(\lambda)q^{|\lambda|}=
q^n(q^{n+1};q)_{\infty}\left(a\left((n)\right) - a(\emptyset)+\sum_{\operatorname{sm}(\gamma){\geq}^* n}\left[a(\gamma\cdot(n))-a(\gamma)\right] q^{|\gamma|}\right).\end{flalign*}
The parenthetical term on the right is exactly $\widetilde{A}_n(q)$, completing the proof. \end{proof}

This lemma works out nicely when the choice of $a(\lambda)$ allows one to factor out aspects of $\widetilde{A}_n(q)$, or simplify the sum otherwise.
Now the following general lemma, wrapping together the preceding formulas, will lead to the proof of Theorem \ref{Thm1} as $q\to 1$.

\begin{lemma}\label{convolemma4}% \label{qgeneral2}
%Let $\operatorname{lg}(\lambda)$ denote the largest part of partition $\lambda$.  
For $a(\lambda)$ a partition-theoretic function, $f(n)$ an arithmetic function, we have 
\begin{flalign*}
\sum_{\lambda\neq \emptyset}(\mu_{\mathcal P}*a)(\lambda)f\left(\operatorname{sm}(\lambda)\right) q^{|\lambda|}  \  =\   (q;q)_{\infty}\sum_{n\geq 1 }\frac{f(n)\widetilde{A}_n(q) q^n}{(q;q)_n},\end{flalign*}
where 
$\widetilde{A}_n(q):=A_n(q)-a(\emptyset)$ with $A_n(q)$ as in \eqref{A_ndef}.
\end{lemma}

\begin{proof}
Multiply Lemma \ref{convolemma2} through by $f(n)$, and sum over $n\geq 1$, noting for each index $n$ the partitions generated on the left-hand side have smallest part $\operatorname{sm}(\lambda)=n$. Swapping order of summation on the left, then, takes $f(n)\mapsto f\left(\operatorname{sm}(\lambda) \right)$. On the right-hand side, use the factorization $(q^{n+1};q)_{\infty}=\frac{(q;q)_{\infty}}{(q;q)_n}$ to give the Lemma \ref{convolemma4}.
\end{proof}

\subsection{Proof of Theorem \ref{Thm1} and Corollaries \ref{Cor2} and \ref{Cor}}

We now have all the ingredients to prove Theorem \ref{Thm1} and its corollaries.

\begin{proof}[Proof of Theorem \ref{Thm1}]
Consider the identity in Lemma \ref{convolemma4} above. By hypothesis we assume $a(\emptyset)=1$.\footnote{A slightly more general theorem follows from using arbitrary $a(\emptyset)$ in this proof.} 
Letting $\chi_{\mathcal S}(n)$ denote the indicator function for $\mathcal S \subseteq \mathbb N$, make the substitution $f(n)\mapsto \frac{\chi_{\mathcal S}(n)}{f(n)}$ (hence the condition that $f$ not vanish on $\mathcal S$ in the denominator) to give
\begin{flalign}\label{convoproof}
\sum_{\operatorname{sm}(\lambda)\in \mathcal S} \frac{(\mu_{\mathcal P}*a)(\lambda)}{f\left(\operatorname{sm}(\lambda)\right)} q^{|\lambda|}  \  =\   (q;q)_{\infty}\sum_{n\in \mathcal S}\frac{\widetilde{A}_n(q) q^n}{f(n)\cdot (q;q)_n}.
\end{flalign}
As $q\to 1$, if $A_n(q)\to 1$ for $n\in \mathcal S$, then $\widetilde{A}_n(q)=A_n(q)-1\to 0$ in the summands on the right. Clearly the right side vanishes as $q\to 1$, which proves the first identity of the theorem.

Similarly, if $A_n(q)\to 1+f(n)$ for $n\in \mathcal S$, then $\widetilde{A}_n(q) \to f(n)$. In the limit, $\frac{\widetilde{A}_n(q)}{f(n)}\to 1$ and the right-hand side of \eqref{convoproof}  asymptotically equals $F_{\mathcal S}(q)$, which approaches $d_{\mathcal S}$ by Theorem \ref{qgeneral}. This is the content of the second identity and its analytic condition.
%
%
%%, for $q$-commensurate $\mathcal S\subseteq \mathbb N$ and $a(\lambda)$ a partition-theoretic function with $a(\emptyset)=1$, we have 
%\begin{flalign}
%\lim_{q\to 1}\sum_{\operatorname{sm}(\lambda)\in \mathcal S}\frac{(\mu_{\mathcal P}*a)(\lambda)}{f\left(\operatorname{sm}(\lambda)\right)}q^{|\lambda|}  \  &=\  -a(\emptyset)\lim_{q\to 1} (q;q)_{\infty}\sum_{n\in \mathcal S}\frac{ q^n}{(q;q)_n} +   \lim_{q\to 1} (q;q)_{\infty}\sum_{n\in \mathcal S}\frac{A_n (q) q^n}{(q;q)_n}\\ \nonumber &=\  -a(\emptyset) \cdot d_{\mathcal S} \  +\  \lim_{q\to 1} (q;q)_{\infty}\sum_{n\in \mathcal S}\frac{A_n (q) q^n}{(q;q)_n},\end{flalign}
%using Theorem \ref{qgeneral} in the final step. The right-hand limit vanishes if and only if 
%\begin{equation}
%\sum_{n\in \mathcal S}\frac{A_n (q) q^n}{(q;q)_n}=o\left(\sum_{n\geq 1}\frac{q^n}{(q;q)_n} \right)
%\end{equation}
%as $q\to 1$, noting $\sum_{n\geq 1}\frac{q^n}{(q;q)_{n}}=(q;q)_{\infty}^{-1}-1$. A sufficient condition for this vanishing is that $|A_n(q)|=o(1)$ as $q\to 1$ (better conditions can likely be found). If $a(\emptyset)=1$, then multiplying through by $-1$ gives the theorem. 
\end{proof}

\begin{remark}
The idea is to set $f(n)=\lim_{q\to 1} \widetilde{A}_n(q)$ for $n\in \mathcal S$, %with $f$ usually {not} a function of $q$, 
so the right-hand side becomes $F_{\mathcal S}(q)$ in the limit.
\end{remark}

\begin{proof}[Proof of Corollary \ref{Cor2}]
We have by hypothesis $A_n(q)= \widetilde{A}_n(q)+1\to 0$, which is equivalent to having $\lim_{q\to 1}\widetilde{A}_n(q)=-1$. Then setting $f(k)=-1$ in \eqref{convoproof} gives the corollary as $q\to 1$. \end{proof}

\begin{proof}[Proof of Corollary \ref{Cor}]
With $a,f,g$ and $b=b_{a,f,g}$ as defined in Corollary \ref{Cor}, following Wang's treatment, apply the partition M\"{o}bius inversion formula \eqref{Pmobius} to the definition of $b_{a,f,g}(\lambda)$ in the statement of the corollary to give
\begin{flalign}\label{Wanalog}
\left(\mu_{\mathcal P}*a\right)(\lambda)\frac{g\left(\operatorname{sm}(\lambda)\right)}{f\left(\operatorname{sm}(\lambda)\right)}\  =\  \sum_{\delta | \lambda}\mu_{\mathcal P} (\delta)b_{a,f,g}(\lambda/\delta)\  =\  (\mu_{\mathcal P}*b)(\lambda).
\end{flalign}
Divide through by $g\left(\operatorname{sm}(\lambda)\right)$ (thus the condition that it also is non-vanishing), multiply both sides by $q^{|\lambda|}$, sum over partitions $\lambda$ with $\operatorname{sm}(\lambda)\in \mathcal S$, then compare the result with the $A_n(q)\to 1+f(n)$ case of Theorem \ref{Thm1}, to see % when $f(n)\neq 0$,
\begin{flalign}\label{Wfinal}
\lim_{q\to 1}\sum_{\operatorname{sm}(\lambda)\in \mathcal S}\frac{(\mu_{\mathcal P}*b)(\lambda)}{g\left(\operatorname{sm}(\lambda)\right)}q^{|\lambda|}  \  =\  \lim_{q\to 1}\sum_{\operatorname{sm}(\lambda)\in \mathcal S}\frac{(\mu_{\mathcal P}*a)(\lambda)}{f\left(\operatorname{sm}(\lambda)\right)}q^{|\lambda|}  \  =\   d_{\mathcal S}.\end{flalign}
\end{proof}

\begin{proof}[Proof of Corollary \ref{Cor5}]
With $b(\lambda)=b_{a,-\varphi, -1}(\lambda), B_n(q)$ as defined in Corollary \ref{Cor5}, if $B_n(q)\to 0$ on $\mathcal S$, the left-hand side of the resulting equation from \eqref{Wfinal} is the instance of Corollary \ref{Cor2} where $a$ is replaced by $b=b_{a,-\varphi,-1}$. The right-hand equality of \eqref{Wfinal} then yields Corollary \ref{Cor5}. \end{proof}

\section{Further examples}\label{ExSect}
In this section we give further examples of applications of the formulas and methods developed above. We note in passing for any $q$-commensurate subset $\mathcal S$, that L'Hospital's rule yields infinitely many formulas for $\lim_{q\to 1}$ $d_{\mathcal S}(q)$ by taking derivatives of the numerator and denominator of \eqref{qdens}. %implies for every $k\geq 1$ that
%\begin{equation}
%\lim_{q\to 1}\frac{\frac{\text{d}^k}{\text{d}q^k}\sum_{n\in \mathcal S}q^n}{\frac{\text{d}^k}{\text{d}q^k}\sum_{n\geq 0}q^n}\  =\  d_{\mathcal S}.
%\end{equation}
Taking antiderivatives instead, then applying L'Hospital's rule, yields another arithmetic density relation for $q$-commensurate subset $\mathcal S$:
 \begin{flalign}
-\lim_{q\to 1} \frac{1}{\operatorname{log}(1-q)}\sum_{n\in \mathcal S}\frac{q^n}{n}\  =\   d_{\mathcal S},
\end{flalign}
 where $\operatorname{log} z$ denotes the principle branch of the complex logarithm function for $z\in \mathbb C$. % (or take $q\in (0,1)$). 

The asymptotic \eqref{qasymp} is suggestive of further, less trivial formulas to compute $d_{\mathcal S}$.

%
%\begin{proof}
%Take $q\mapsto q^k$ in \eqref{qasymp} and sum both sides over $k\geq 1$. %Noting $0\leq d_{\mathcal S} \leq 1$, 
%Swap order of %(absolutely convergent) 
%summation $\sum_{k\geq 1}\sum_{n\in \mathcal S}q^{nk}=\sum_{n\in \mathcal S}\frac{q^n}{1-q^n}\   \sim\  d_{\mathcal S} \sum_{k\geq 1} \frac{q^k}{1-q^k}$; Tannery's theorem justifies factoring out $d_{\mathcal S}=\lim_{q\to 1}d_{\mathcal S}(q^k)$ on the right. A little algebra completes the proof. 
%\end{proof}
%
%
%
%\begin{proof}
%%We proceed almost exactly as in the previous proof. 
%Take $q\mapsto q^k$ in \eqref{qasymp}. Multiply both sides of the asymptotic equality by $\mu(k)$, sum over $k\geq 1$, and swap order of summation on the left to arrive at% (appealing again to Tannery's theorem to factor out $d_{\mathcal S}$) we have  % (allowing $\lim_{q\to 1}$ inside the sum over $k$) and Lambert series:
%\begin{equation}
%\sum_{k\geq 1}\sum_{n\in \mathcal S}\mu(k)q^{nk}\  \sim\   d_{\mathcal S} \cdot \sum_{k\geq 1}\frac{\mu(k)q^k}{1-q^k}
%\nonumber \  =\  d_{\mathcal S}\cdot \sum_{k\geq 1}q^{k}\sum_{d|k}\mu(d)\  =\  d_{\mathcal S} \cdot q, %\  \to \  d_{\mathcal S}% \lim_{q\to 1}q=d_{\mathcal S}.
%\end{equation}
%%using classical identities. 
%which approaches $d_{\mathcal S}$ as $q\to 1$. \end{proof}
%
%Examples \ref{ex1} and \ref{muthm} are special cases of a general relation. 

\begin{example}\label{genthm} Let $\mu(k)$ denote the classical M\"{o}bius function, and let $f(k)$ be an arithmetic function that is not identically zero.  For $ \mathcal S \subseteq \mathbb N$ a $q$-commensurate subset, we have 
\begin{flalign*}
\lim_{q\to 1}\  \frac{\sum_{n\in \mathcal S}\sum_{k\geq 1}(\mu*f)(k)q^{nk }}{\sum_{k\geq 1} {f(k)q^k}}\  =\   d_{\mathcal S},
\end{flalign*}
where $*$ denotes classical Dirichlet convolution.
\end{example}

\begin{proof}
Take $q\mapsto q^k$ in \eqref{qasymp}. Multiply both sides of the asymptotic equality by $(\mu*f)(k)=\sum_{d|k}\mu(d)f(k/d)$, sum over $k\geq 1$, and (noting the series are absolutely convergent) swap order of summation on the left. Then as $q\to 1$, well-known relations involving classical Lambert series and Dirichlet convolution give  
\begin{equation}\label{eq}
\sum_{n\in \mathcal S} \sum_{k\geq 1} (\mu*f)(k)q^{nk}\  \sim \  d_{\mathcal S}\cdot \sum_{k\geq 1}\frac{(\mu*f)(k)q^k}{1-q^k}\  =\  d_{\mathcal S}\cdot \sum_{k\geq 1}f(k)q^{k}.
\end{equation} %We note that Tannery's theorem justifies factoring out $d_{\mathcal S}=\lim_{q\to 1}d_{\mathcal S}(q^k)$ on the right side of the asymptotic. 
Dividing through by $\sum_{k\geq 1}f(k)q^{k}$ completes the proof.
\end{proof}

Setting $f(k):=1$ for all $k\geq 1$, then Example \ref{genthm} reduces to \eqref{qdream}. Setting $f(k):=\sigma_0(k)$, the classical divisor function such that $(\mu * \sigma_0)(k)=1$, yields a Lambert series calculation. % example. % the first example above. 

\begin{example}\label{ex1} For $ \mathcal S \subseteq \mathbb N$ a $q$-commensurate subset, we have 
\begin{equation*} 
\lim_{q\to 1}\frac{\sum_{n\in \mathcal S}\frac{q^n}{1-q^n}}{\sum_{n\geq 1} \frac{q^n}{1-q^n}}\  =\  d_{\mathcal S}.
\end{equation*}
\end{example}

Setting $f(1):=1, f(k):=0$ for $k> 1$, viz. the function $f(k):=\sum_{d|k}\mu(d)$, gives another density formula that resembles a power series analogue of the first sum in \eqref{F_S}.

\begin{example}\label{muthm}  %Let $\mu(k)$ denote the classical M\"{o}bius function. 
For $ \mathcal S \subseteq \mathbb N$ a $q$-commensurate subset, we have 
\begin{flalign*}
\lim_{q\to 1}\sum_{n\in \mathcal S}\sum_{k\geq 1}\mu(k)q^{nk}\  =\  d_{\mathcal S}.% \lim_{q\to 1}q=d_{\mathcal S}.
\end{flalign*}
\end{example}

Setting $f(n)$ equal to the partition function $p(n)$ in Example \ref{genthm}, gives an example that looks more like the formulas in \cite{Paper1} and the present work. 
 
\begin{example} Let $p(n)$ denote the partition function.  For $ \mathcal S \subseteq \mathbb N$ a $q$-commensurate subset, we have 
\begin{flalign*}
\lim_{q\to 1}\  (q;q)_{\infty}\sum_{n\in \mathcal S}\sum_{k\geq 1}(\mu*p)(k)q^{nk }\  =\   d_{\mathcal S}.
\end{flalign*}
\end{example}
%
%\begin{proof}
%Take $f(n)=p(n)$ in Example \ref{genthm}, noting $\sum_{n\geq 1}p(n)q^n \sim (q;q)_{\infty}^{-1}$ as $q\to 1$.
%\end{proof}
%

%By well known Lambert series and Dirichlet convolution relations, t
%The case $f(n):=d(n)$, the classical divisor function, gives Example \ref{ex1} above. The case $f(n):=1$ if $n=1$ and $f(n):=0$ otherwise, %(i.e., $f(n)=\sum_{d|n}\mu(d)$) 
%gives Example \ref{muthm}. 

Beyond giving analogues of classical convolution theorems as in Section \ref{Sect2} above, and giving formulas for expected value and arithmetic density, %combined with generating function techniques, easy 
partition-theoretic convolution leads to other interesting $q$-series relations. Here is a nice consequence of Lemma \ref{convolemma}, making another connection between partition convolution and modular forms. % we give two examples.
%In the following example, partition convolution induces an averaging phenomenon with modular properties.
%
%In the next example, partition convolution gives a version of the $q$-binomial theorem.
%
%
%\begin{example}\label{qbracketcor2}% \label{qgeneral2}
%%Let $\operatorname{lg}(\lambda)$ denote the largest part of partition $\lambda$.  
%For $|q|<1$,  we have 
%$$
%\sum_{\lambda \in \mathcal P}(\mu_{\mathcal P}*z^{\ell})(\lambda)q^{|\lambda|}=\sum_{n\geq 0}\frac{q^n (z;q)_n}{(q;q)_n}.$$
%% an arithmetic function with $f(0):=1$, we have 
%\end{example}
%\begin{proof} This is a rewriting of $(q;q)_{\infty}(z;q)_{\infty}^{-1}$ via \eqref{Cauchy2}, and can also be proved from \eqref{convolemma} under the substitution $a(\lambda)=\mu_{\mathcal P}(\lambda), b(\lambda)=z^{\ell}(\lambda):=z^{\ell(\lambda)}$, with a little algebra.
%\end{proof}
%
%

\begin{example}\label{cor1}% \label{qgeneral2}
%Let $\operatorname{lg}(\lambda)$ denote the largest part of partition $\lambda$.  
For $z\in \mathbb C, z\neq 0,$ set $z^{\ell}(\lambda):=z^{\ell(\lambda)},\  z^{-\ell}(\lambda):=z^{-\ell(\lambda)}$. Recall the rank generating function for strongly unimodal sequences, $U(z,q):=\sum_{n\geq 0}q^{n+1}(-zq;q)_{n}(-z^{-1}q;q)_{n}$. Then  
$$
\sum_{\lambda \in \mathcal P}(z^{\ell}*z^{-\ell})(\lambda)q^{|\lambda|}\  =\  1+\frac{z+z^{-1}}{(zq;q)_{\infty}(z^{-1}q;q)_{\infty}}U(-z,q).$$
% an arithmetic function with $f(0):=1$, we have 
\end{example}

\begin{remark}
We note in Example \ref{cor1} a close relation to the crank generating function of Andrews-Garvan, $C(z,q)=\frac{(q;q)_{\infty}}{(zq;q)_{\infty}(z^{-1}q;q)_{\infty}}$ (see  \cite{A-G}), viz. 
\begin{equation}
\sum_{\lambda \neq \emptyset}(z^{\ell}*z^{-\ell})(\lambda)q^{|\lambda|}\  =\frac{(z+z^{-1})C(z,q) U(-z,q)}{(q;q)_{\infty}}.
\end{equation}
\end{remark}

\begin{proof} Take $a=z^{\ell}, b=z^{-\ell}$ in Lemma \ref{convolemma} to yield
\begin{flalign*}
\sum_{\operatorname{sm}(\lambda)=n}(z^{\ell}*z^{-\ell})(\lambda)q^{|\lambda|}&=\frac{zq^n}{(zq^{n};q)_{\infty}}\cdot \frac{1}{(z^{-1}q^{n};q)_{\infty}}+\frac{1}{(zq^{n};q)_{\infty}}\cdot \frac{z^{-1}q^n}{(z^{-1}q^{n};q)_{\infty}}\\\nonumber &=\frac{z+z^{-1}}{(zq;q)_{\infty}(z^{-1}q;q)_{\infty}}\cdot q^{n}(zq;q)_{n-1}(z^{-1}q;q)_{n-1}.
\end{flalign*}
Summing over $n\geq 1$, then adding 1 to both sides of the equation to adjoin the $\lambda=\emptyset$ term on the left, and comparing with the definition of $U(z,q)$, completes the proof. 
%One can prove the lemma combinatorially by replicating the proof of Theorem \ref{qgeneral} above, but weighting each part $\gamma$ in the proof by an additional multiplicative factor $f(n)$, with $n\in \mathcal S$ the multiplicity of the largest part of $\gamma$, which maps to $f(\operatorname{sm}(\lambda))$ under conjugation. 
\end{proof}

%It is interesting in Example \ref{cor1} that partition convolution on the left side generates, in a sense, unimodal sequences on the right. 
%Beyond the relation to unimodal sequences in combinatorics, 
%Example \ref{cor1} intersects the theory of modular forms: 
The rank generating function for strongly unimodal sequences is connected to mock modular and quantum modular forms (see e.g. \cite{FOR}). Example \ref{cor1} itself, equal to $(zq;q)_{\infty}^{-1}(z^{-1}q;q)_{\infty}^{-1}$ by \eqref{Cauchy2}, is modular for $z$ a root of unity, up to multiplication by a rational power of $q$. %Multiplying both sides by $(q;q)_{\infty}$ yields the crank generating function $C(z,q)$ of Andrews-Garvan \cite{AndrewsGarvan}. %, which is similarly deeply connected in modularity theory, on top of its relation to Ramanujan congruences.

 One may obtain different forms of partition convolution identities for $d_{\mathcal S}$ from those in Theorem \ref{Thm1} as consequences of Lemma \ref{convolemma4}, such as the following pair of identities.

\begin{example} \label{Cor3}
%Let $\operatorname{lg}(\lambda)$ denote the largest part of partition $\lambda$.  
For $S\subseteq \mathbb N$ a $q$-commensurate subset, $a(\lambda)=\ell(\lambda)$ (partition length) and $\sigma_0(n):=\sum_{d|n}1$ as usual, we have
\begin{flalign*}
\lim_{q\to 1} \frac{\sum_{\operatorname{sm}(\lambda)\in \mathcal S}(\mu_{\mathcal P}*\ell)(\lambda)q^{|\lambda|}}{\sum_{n\geq 1}\sigma_0(n)q^n}\  =\  d_{\mathcal S}.
\end{flalign*}
%Let $\operatorname{lg}(\lambda)$ denote the largest part of partition $\lambda$.  
For $S\subseteq \mathbb N$ a $q$-commensurate subset, $a(\lambda)=\operatorname{sz}(\lambda):=|\lambda|$ (partition size)  and $\sigma_1(n):=\sum_{d|n}d$, we have
\begin{flalign*}
\lim_{q\to 1} \frac{\sum_{\operatorname{sm}(\lambda)\in \mathcal S}(\mu_{\mathcal P}*\operatorname{sz})(\lambda)q^{|\lambda|}}{\sum_{n\geq 1}\sigma_1(n)q^n}  %\  =\  \lim_{q\to 1}\frac{\sum_{\operatorname{sm}(\lambda)\in \mathcal S}\frac{(\mu_{\mathcal P}*\operatorname{sz})(\lambda)}{\operatorname{sm}(\lambda)}q^{|\lambda|}}{\sum_{n\geq 1}\sigma_0(n)q^n}\  
\  =\    d_{\mathcal S}.
\end{flalign*}
\end{example}

%
%\begin{theorem}\label{convolemma2.5}% \label{qgeneral2}
%%Let $\operatorname{lg}(\lambda)$ denote the largest part of partition $\lambda$.  
%For $a(\lambda)$ a partition-theoretic function, we have 
%\begin{flalign*}
%\lim_{q\to 1}\frac{\sum_{\operatorname{sm}(\lambda)\in \mathcal S}(\mu_{\mathcal P}*a)(\lambda) q^{|\lambda|}}{\sum_{\lambda \in \mathcal P}(\mu_{\mathcal P}*a)(\lambda) q^{|\lambda|}}\   =\  d_{\mathcal S}.
%\end{flalign*} \end{theorem}
%
%
%

\begin{proof}
For the first identity, set $a(\lambda)=\ell(\lambda)$ and $f$ equal to the indicator function of $\mathcal S$ in Corollary \ref{convolemma4}. Then since $\ell((n))=1, \  \ell(\emptyset)=0$, and $\ell(\gamma\cdot(n))-\ell(\gamma)=1$ for any $\gamma\in \mathcal P$, we have
\begin{equation}\label{ell}
\widetilde{A}_n(q)=1+\sum_{\operatorname{sm}(\gamma){\geq} n}q^{|\gamma|}=\frac{1}{(q^n;q)_{\infty}}=\frac{(q;q)_{n-1}}{(q;q)_{\infty}}.
\end{equation}
Note from the right-hand side that $\widetilde{A}_n(q)$ does {\it not} satisfy either of the analytic conditions of Theorem \ref{Thm1} as $q\to 1$. However, by \eqref{ell}, it fits nicely into the right-hand side of Lemma \ref{convolemma4}:
\begin{flalign*}
\sum_{\operatorname{sm}(\lambda)\in \mathcal S}(\mu_{\mathcal P}*\ell)(\lambda) q^{|\lambda|}  \  =\  (q;q)_{\infty}\sum_{n\in \mathcal S}\frac{\widetilde{A}_n(q) q^n}{ (q;q)_n}\  =\  \sum_{n\in \mathcal S}\frac{q^n(q;q)_{n-1}}{ (q;q)_n}\  =\  \sum_{n\in \mathcal S}\frac{q^n}{1-q^n},
\end{flalign*}
which reduces the corollary to a rewriting of Example \ref{ex1} above. Dividing through by $\sum_{n\geq 1}\frac{q^n}{1-q^n}=\sum_{n\geq 1}\sigma_0(n)q^n$ and letting $q\to 1$ completes the proof of the first identity. 

We prove the second identity along very similar lines, setting $a(\lambda)=\operatorname{sz}(\lambda)=|\lambda|$. Then since $|(n)|=n, \  |\emptyset|=0$, and $|\gamma\cdot(n)|-|\gamma|=n$ for any $\gamma\in \mathcal P$, in this case we compute 
\begin{equation}\label{sz}
\widetilde{A}_n(q)=n+\sum_{\operatorname{sm}(\gamma){\geq} n}n q^{|\gamma|}=\frac{n}{(q^n;q)_{\infty}}=\frac{n\cdot (q;q)_{n-1}}{(q;q)_{\infty}}.
\end{equation}
Again, this $\widetilde{A}_n(q)$ does {not} satisfy the conditions of Theorem \ref{Thm1}; however, by \eqref{sz}, we have 
\begin{flalign*}
\sum_{\operatorname{sm}(\lambda)\in \mathcal S}(\mu_{\mathcal P}*\operatorname{sz})(\lambda) q^{|\lambda|}  \  =\  (q;q)_{\infty}\sum_{n\in \mathcal S}\frac{\widetilde{A}_n(q) q^n}{ (q;q)_n}\  =\  \sum_{n\in \mathcal S}\frac{nq^n(q;q)_{n-1}}{ (q;q)_n}\  =\  \sum_{n\in \mathcal S}\frac{nq^n}{1-q^n}.
\end{flalign*}
Now, the right-hand side of these equalities can be rewritten 
\begin{flalign}\label{above}
\sum_{n\in \mathcal S}\frac{nq^n}{1-q^n}\  =\  \sum_{n\in \mathcal S}n\sum_{k\geq 1}q^{nk}\  =\  \sum_{k\geq 1}\sum_{n\in \mathcal S}nq^{nk}.\end{flalign}
It follows from \eqref{qasymp} that
\begin{flalign}\label{above2}
\sum_{n\in \mathcal S}nq^{n}\  =\  q\frac{{d}}{{d}q}\sum_{n\in \mathcal S}q^{n}\   \sim\   d_{\mathcal S} \cdot \frac{q}{(1-q)^2},
\end{flalign}
which, combined with \eqref{above} above, gives
\begin{flalign*}
\sum_{n\in \mathcal S}\frac{nq^n}{1-q^n}\  \sim\  d_{\mathcal S}\cdot \sum_{k\geq 1}\frac{q^k}{(1-q^k)^2}\  =\  d_{\mathcal S}\cdot \sum_{k\geq 1}\frac{kq^k}{1-q^k}\  =\  d_{\mathcal S}\cdot \sum_{k\geq 1}\sigma_1(k)q^k,
\end{flalign*}
using well-known Lambert series identities in the right-hand equalities. Dividing through by $\sum_{n\geq 1}\sigma_1(n)q^n$ and letting $q\to 1$ gives the second identity.
\end{proof}

\begin{remark} %It is interesting to observe in Example \ref{Cor3} this correspondence: $\ell(\lambda)$ is the {number} of parts of $\lambda$ in the numerator of the first identity, while $\sigma_0(n)$ is the {number of divisors} of $n$ in the denominator; likewise, $\text{sz}(\lambda)$ is the {\it sum} of the parts in the second identity, whereas $\sigma_1(n)$ is the {sum of divisors}. 
Note that $\ell(\lambda)$ and $\operatorname{sz}(\lambda)$ are the $r=0$ and $r=1$ cases, respectively, of  $$\operatorname{sz}_{r}(\lambda) := \sum_{\lambda_i\in \lambda} \lambda_{i}^{r},$$ where the sum is over the parts $\lambda_i\in \mathbb N$ of partition $\lambda$.\footnote{One might write $\operatorname{sz}_r(\lambda)=|\lambda|_r$ to generalize the standard absolute value notation for partition size.} More generally, inserting $\operatorname{sz}_{r}(\lambda)$ for any integer $r\geq 0$ in place of $\ell(\lambda)$ and $\operatorname{sz}(\lambda)$, the steps of the proof of Example \ref{Cor3} continue to hold, but yielding $(\mu_{\mathcal P}*\operatorname{sz}_{r})(\lambda)$ in the numerator and $\sigma_{r}(n):=\sum_{d|n}d^r$ in the denominator.
\end{remark}

\begin{remark}
Note from \eqref{above2} and from the fact $\frac{q}{(1-q)^2}=\sum_{k\geq 1}\frac{\varphi(k)q^k}{1-q^k}=\sum_{k\geq 1}kq^k$, that 
\begin{equation}
\lim_{q\to 1}\frac{\sum_{n\in \mathcal S}nq^{n}}{\sum_{n\geq 1}nq^n}=d_{\mathcal S},
\end{equation}
which can also be found by applying L'Hospital's rule to \eqref{qdens} as $q\to 1$, as noted above.
\end{remark}

We anticipate that formulas similar to those proved here will hold for sums over partitions with parts restricted in other ways to subsets $\mathcal S$ such as in Corollary \ref{qgeneral2}, as well as for $q\to \zeta$ a root of unity not equal to one, along the lines of those in \cite{Paper1}; but pursuing this was outside of our present aims. Using limiting techniques such as the Hardy-Ramanujan circle method \cite{HR}, techniques of Watson \cite{Watson} and other methods from the universe of $q$-series and modular forms --- as well as limiting values of special functions at roots of unity --- perhaps closed formulas can be computed for $d_{\mathcal S}$ beyond the cases of $\mathcal S$ being integers in arithmetic progressions, $k$th power-free integers, and other subsets of $\mathbb N$ with arithmetic densities already explicitly known.

\section*{Acknowledgments}
We wish to express our gratitude to A. V. Sills for computing the values in Table 2, and to Paul Pollack for providing advice and references on Tauberian theorems. Moreover, the second author thanks G. E. Andrews and Prof. Sills for discussions on convergence, divergence and $q$-analogues that strongly informed this work.

\end{document}